%% file: GenKazh.tex
\documentclass[12pt]{amsart}
\usepackage{amscd,amsmath,amsthm,amssymb,mathtools}
\usepackage{mathrsfs}
\usepackage{comment}
\usepackage[all,cmtip]{xy}
\usepackage{color}

\definecolor{cadmiumgreen}{rgb}{0.0, 0.42, 0.24}
\usepackage[
colorlinks, citecolor=cadmiumgreen,
pagebackref,
pdfauthor={Hyungryul Baik, Farbod Shokrieh, Chenxi Wu}, 
pdftitle={Limits of canonical forms on towers of Riemann surfaces}, 
pdfstartview ={FitV},
]{hyperref}

\usepackage[
alphabetic,
backrefs,
msc-links,
nobysame,
lite,
]{amsrefs} 

\usepackage{tikz, float} \usetikzlibrary {positioning}

\usetikzlibrary{calc,decorations.markings}
\usetikzlibrary{shapes,snakes}

\usepackage[left=3.5cm,top=3.5cm,right=3.5cm]{geometry}

\def\ZZ{{\mathbb Z}}
\def\RR{{\mathbb R}}
\def\CC{{\mathbb C}}

\def\B{{\mathcal B}}

\def\A{{\mathcal A}}

\def\M{{\mathcal M}}
\def\H{{\mathcal H}}

\def\G{{\mathbf G}}

\def\opn#1#2{\def#1{\operatorname{#2}}} 

\opn\depth{depth} 
\opn\codim{codim}
\opn\ini{in} 
\opn\LM{LM}
\opn\LC{LC}
\opn\NF{NF}
\opn\Merge{Merge}
\opn\sgn{sgn}
\opn\div{div} 
\opn\Div{Div} 
\opn\Pic{Pic}
\opn\Prin{Prin}
\opn\Del{Del}
\opn\op{op}
\opn\ends{ends}
\opn\indeg{indeg} 
\opn\outdeg{outdeg}
\opn\red{red}
\opn\Spec{Spec} 
\opn\Supp{Supp} 
\opn\Meas{Meas}
\opn\supp{supp} 
\opn\Ker{Ker} 
\opn\Coker{Coker} 
\opn\sign{sign}
\opn\Hom{Hom}
\opn\Tor{Tor} 
\opn\id{id}
\opn\cl{cl}
\opn\span{span}
\opn\trace{trace}
\opn\Image{Image}
\opn\con{conv} 
\opn\relint{rel.int} 
\opn\vol{vol}
\opn\im{im}
\opn\val{val}
\opn\Zh{Zh}
\opn\Ber{Ber}
\opn\can{can}
\opn\syz{{\rm syz}}
\opn\spoly{{\rm spoly}}
\opn\LM{{\rm LM}}
\opn\lm{{\rm lm}}
\opn\lcm{{\rm lcm}} 
\opn\Tr{{\rm Tr}}
\opn\A{\mathcal A}
\opn\L{L}
\opn\dist{dist}
\opn\pd{pd}
\opn\en{en}

\opn\H{\mathrm{H}}
\opn\V{\mathrm{V}}
\opn\B{\mathcal{B}}
\opn\Tr{\mathrm{Tr}}
\opn\Gr{\mathbf{Gr}}

\def\Implies{\ifmmode\Longrightarrow \else
        \unskip${}\Longrightarrow{}$\ignorespaces\fi}
\def\implies{\ifmmode\Rightarrow \else
        \unskip${}\Rightarrow{}$\ignorespaces\fi}
\def\iff{\ifmmode\Longleftrightarrow \else
        \unskip${}\Longleftrightarrow{}$\ignorespaces\fi}

\newtheorem{Theorem}{Theorem}[section]
\newtheorem{Lemma}[Theorem]{Lemma}

\newtheorem{Proposition}[Theorem]{Proposition}

\theoremstyle{remark}
\newtheorem{Remark}[Theorem]{Remark}

\theoremstyle{definition}

\newtheorem{Definition}[Theorem]{Definition}

\def\qed{\ifhmode\textqed\fi
      \ifmmode\ifinner\quad\qedsymbol\else\dispqed\fi\fi}
\def\textqed{\unskip\nobreak\penalty50
       \hskip2em\hbox{}\nobreak\hfil\qedsymbol
       \parfillskip=0pt \finalhyphendemerits=0}
\def\dispqed{\rlap{\qquad\qedsymbol}}

\tikzstyle{Cwhite}=[scale = .8,circle, fill = white, minimum size=3mm] 
\tikzstyle{Cgray}=[scale = .4,circle, fill = gray, minimum size=3mm] 
\tikzstyle{Cblack2}=[scale = .4,circle, fill = black, minimum size=5mm] 
\tikzstyle{Cblack}=[scale = .7,circle, fill = black, minimum size=3mm]
\tikzstyle{C0}=[scale = .9,circle, fill = black!0, inner sep = 0pt, minimum size=3mm]
\tikzstyle{C1}=[scale = .7,circle, fill = black!0, inner sep = 0pt, minimum size=3mm]
\tikzstyle{Cred}=[scale = .4,circle, fill = red, minimum size=3mm]

\title[Limits of canonical forms on towers of Riemann surfaces]{Limits of canonical forms on towers of Riemann surfaces}

\author{Hyungryul Baik}
\address{Department of Mathematical Sciences, KAIST,  
291 Daehak-ro, Yuseong-gu, Daejeon 34141, South Korea}
\email{\href{mailto:hrbaik@kaist.ac.kr}{hrbaik@kaist.ac.kr}}

\author {Farbod Shokrieh}
\address{Department of Mathematical Sciences, University of Copenhagen\\
Universitetsparken 5, 2100 K\o benhavn \O \\ 
Denmark
}
\email{\href{mailto:farbod.shokrieh@gmail.com}{farbod.shokrieh@gmail.com}}

\author{Chenxi Wu}
\address{Department of Mathematics, Rutgers University\\
Piscataway, New Jersey 08854-8019\\USA}
\email{\href{mailto:wuchenxi2013@gmail.com}{wuchenxi2013@gmail.com}}

\subjclass[2010]{
\href{https://mathscinet.ams.org/msc/msc2010.html?t=14H55}{14H55}, 
\href{https://mathscinet.ams.org/msc/msc2010.html?t=14H30}{14H30}, 
\href{https://mathscinet.ams.org/msc/msc2010.html?t=28A33}{28A33}, 
\href{https://mathscinet.ams.org/msc/msc2010.html?t=35J05}{35J05},
\href{https://mathscinet.ams.org/msc/msc2010.html?t=58A10}{58A10}}

\date{\today}

\begin{document}

\begin{abstract}
We prove a generalized version of Kazhdan's theorem for canonical forms on Riemann surfaces.
In the classical version, one starts with an ascending sequence $\{S_n \rightarrow S\}$ of finite Galois covers of a hyperbolic Riemann Surface $S$, converging to the {\em universal cover}. The theorem states that the sequence of forms on $S$ inherited from the canonical forms on $S_n$'s converges uniformly to (a multiple of) the hyperbolic form.
We prove a generalized version of this theorem, where {\em the universal cover} is replaced with {\em any infinite Galois cover}. Along the way, we also prove a Gauss--Bonnet type theorem in the context of arbitrary infinite Galois covers.
\end{abstract}

\maketitle

\setcounter{tocdepth}{1}
\tableofcontents

\section{Introduction}\label{sec:introduction}
\subsection{Background}

A compact connected Riemann surface $S$ of genus $g \geq 2$ can be given a {\em canonical (Arakelov) $(1,1)$-form} by embedding the surface inside its Jacobian via the Abel-Jacobi map, and pulling back the canonical (`Euclidean') translation-invariant $(1,1)$-form.
A celebrated theorem of Kazhdan (\cite[\S3]{Kaj}, see also \cite{mumford, Kazhdan}) states that `the hyperbolic form is the limit of the canonical forms': if $\{S_n \rightarrow S\}$ is an ascending sequence of finite Galois covers converging to the {\em universal cover}, then the $(1,1)$-forms on $S$ inherited from canonical forms via $S_n \rightarrow S$ converge uniformly to a multiple of the hyperbolic $(1,1)$-form. See \cite{rhodes1993sequences} and \cite[Appendix]{mcmullen2013entropy} for two different proofs of this result. See also \cite{Yau, Donnelly, ohsawa2009remark, Ohsawa2, Treger, Yeung, ChenFu} for various related results and generalizations. The purpose of this work is to prove a more general version of this theorem, where {\em the universal cover} is replaced with {\em any infinite Galois cover}. Our results also make sense in genus $g=1$.

It is known that the canonical form on the Poincar\'e unit disk is the same as the hyperbolic form (up to a constant multiple). So Kazhdan's theorem can be restated, informally, as follows: the limit of the induced forms from finite Galois covers coincides with the induced form from the limiting space (the universal cover). This is the statement that we generalize. 

\subsection{Our results}

The main purpose of this paper is to prove the following.

\vspace{3mm}
\noindent{\bf Theorem~A. (A generalized Kazhdan theorem)} 
Let $S'$ be {\em any infinite Galois cover} of a compact connected Riemann surface $S$ of genus $g \geq 1$. Let $\{S_n \rightarrow S\}$ be a sequence of finite Galois covers converging to $S'$. Then the sequence of $(1,1)$-forms on $S$ induced from the canonical forms on $S_n$ {\em converges uniformly} to the $(1,1)$-form induced from the canonical form on $S'$.

\vspace{1mm}
See Theorem~\ref{thm:arakelov conv} for a more precise statement. 
\vspace{2mm}

Most of our work goes into the proof of a weaker result, stating that the associated measures attached to $(1,1)$-forms converge strongly (Theorem~\ref{thm:Kazhdan_surface}). The uniform convergence of forms will then follow from a standard (but subtle) analytic argument.

We first directly define and study the notion of canonical measures on Riemann surfaces (Definition~\ref{def:canmeas}). These are precisely the measures attached to canonical $(1,1)$-forms (Lemma~\ref{lem:measvsform}), but they can be defined in more general situations (e.g. even if the surface is not orientable -- see Remark~\ref{rmk:canmeas} (iii)). The main reason we study these measures directly is that they are closely related to the operator theory on various natural Hilbert spaces attached to Riemann surfaces (see, e.g., Proposition~\ref{prop:cm_projection_formula}). This will allow us to use powerful techniques from operator theory and the theory of von Neumann algebras.

A fundamental property of the hyperbolic measure on a compact Riemann surface is the well-known consequence of the Gauss--Bonnet formula, stating that the hyperbolic volume of a Riemann surface $S$ has a simple expression in terms of the Euler characteristic of $S$. Our second main theorem states that {\em all limiting measures} (coming from any infinite Galois cover) satisfy a similar Gauss--Bonnet type property.

\vspace{3mm}
\noindent{\bf Theorem~B. (A generalized Gauss--Bonnet)} 
Let $\phi:S' \to S$ be an infinite Galois covering of a compact connected Riemann surface $S$ of genus $g \geq 1$, and let $\mu$ be the measure on $S$ induced from the canonical measure on $S'$. Then 
\[\mu(S)= -\frac{1}{2} \chi(S)= g-1 \, .\]

\vspace{1mm}
See Theorem~\ref{thm:GaussBonnet} and Remark~\ref{rmk:L2betti2}. 
\vspace{2mm}

Indeed, Theorem~B is also a crucial ingredient in the proof of convergence of canonical measures (Theorem~\ref{thm:Kazhdan_surface}). Our proof shows that one may think of the equality $\mu(S)=g-1$ as a `trace formula'.
For the proof of Theorem~B, we will need a slight variant of L\"uck's approximation theorem (\cite[Theorem 0.1]{Luck2}) to deal with general infinite Galois covers (Theorem~\ref{thm:luck_l2dim}).

\subsection{Related work and directions}
Kazhdan's theorem is usually stated in terms of the {\em canonical metric}, i.e. the metric $ds^2$ whose associated $(1,1)$-form $-{\rm Im}\,{ds^2}/2$ is our canonical $(1,1)$-form. This is a hermitian metric on the holomorphic tangent bundle on the Riemann surface.
It is well-known that such a metric is directly recovered from its associated $(1,1)$-form (see, e.g., \cite[pp 28--29]{GriffHar}). As we mostly work with the corresponding measure, we found it more convenient to state and prove our results in the language of the canonical forms. See also Remark~\ref{rmk:metricsVSforms} for some related notions and terminology in the literature. 

Metric graphs may be considered as tropical (or non-Archimedean) analogues of Riemann surfaces. For metric graphs, analogues of Theorem~A and Theorem~B are proved in \cite{ShokriehWu} by the second and third named authors. Following the discussion in \cite[\S7]{ShokriehWu}, we expect that the limiting measure inherited from the `Schottky cover' of a Riemann surface to be closely related to other notions such as Poisson-Jensen and equilibrium measures. The limiting measure inherited from the maximal abelian cover should also have a nice intrinsic interpretation.

The starting point for our approach was the very slick proof of the classical Kazhdan's theorem given by Curtis McMullen in \cite[Appendix]{mcmullen2013entropy} (based on the original argument in \cite[\S3]{Kaj}). The main difficulty in our generalization is the basic fact that one knows explicitly the limiting measure on the universal cover (i.e. the hyperbolic measure) whereas, for an arbitrary infinite Galois cover, one does not have such explicit knowledge. Our use of $L^2$ and von Neumann algebraic techniques is to overcome this difficulty.

It is conceivable that our work can be generalized in various directions; there could be Kazhdan-type theorems for other types of limits of Riemann surfaces, most importantly Benjamini--Schramm limits in the sense of \cite{ABBGNRS}.

\subsection{Structure of the paper}
In \S\ref{sec:NotBack} we review some basic facts and set the terminology and notations.
In \S\ref{sec:canonical_measure} we give the definitions for canonical measures and forms, and establish some of their basic properties. Our emphasis is to give operator theoretic interpretations and formulas.
In \S\ref{sec:l2approx} we first review some basics from the theory of von Neumann algebras and dimensions, as well as the $L^2$ Hodge--de Rham theory. The main purpose is to prove a variant of L\"uck's approximation theorem (Theorem~\ref{thm:luck_l2dim}).
In \S\ref{sec:proof_Kazhdahn_surface} we prove our main results (Theorem~A and Theorem~B).

\subsection*{Acknowledgement}
We would like to thank John Hubbard, Bingbing Liang, Curtis McMullen, and Nicolas Templier for their interest in this project, and for helpful remarks and conversations. 
We would also like to thank Curtis McMullen, Nicolas Templier, and the anonymous referee for their comments on an earlier draft. 
The first author was partially supported by Samsung Science \& Technology Foundation grant no. SSTF-BA1702-01.

\section{Notation and Background} \label{sec:NotBack}

\subsection{$L^2$ forms on Riemann surfaces}
Throughout, by a Riemann surface, we will mean a (possibly non-compact) {\em connected} Riemann surface. Our compact Riemann surfaces will always have genus $g \geq 1$.

Let $S$ be a Riemann surface, and let $\Omega_{L^2}(S)$ denote the Hilbert space completion of $\Omega^1_c(S)$, the space of complex (global) $1$-forms with compact support endowed with the hermitian inner product
\begin{equation}\label{eq:innerprod}
\langle \alpha, \beta\rangle= \int_S \alpha\wedge \star\bar{\beta }\, ,
\end{equation}
where $\star$ is the Hodge star operator. As usual, we use the notation $\| \alpha \|_{L^2}^2 = \langle \alpha , \alpha \rangle$. 

Recall, on a Riemann surface, the Hodge star operator is defined by the local formula
\begin{equation} \label{star_eq}
\star dz= -idz \quad , \quad \star \, d\bar{z}= id\bar{z} \, .
\end{equation}
In particular, it depends only on the complex structure and not on the choice of the Riemannian metric. 

\begin{Remark}\label{rmk:holant}
The space of holomorphic $1$-forms ({\em differentials of the first kind}) $\Omega_{L^2}^{1,0}(S) \subset \Omega_{L^2}(S)$ and antiholomorphic $1$-forms $\Omega_{L^2}^{0,1}(S) \subset \Omega_{L^2}(S)$ are orthogonal under the inner product in \eqref{eq:innerprod}. For $\alpha, \beta\in\Omega_{L^2}^{1,0}(S)$, one computes
\[
\langle \alpha, \beta \rangle = \langle \bar{\alpha}, \bar{\beta} \rangle = i \int_S \alpha\wedge \bar{\beta} = 2 \, (\alpha, \beta) \, 
\]
where $(\cdot , \cdot)$ is the usual {\em Hodge inner product} on $\Omega_{L^2}^{1,0}(S)$.
\end{Remark}

\begin{Remark}\label{rmk:separable}
The Hilbert space $\Omega_{L^2}(S)$ is {\em separable}. This is because the Riemann surface $S$ is second-countable and, consequently, the Borel $\sigma$-algebra is countably generated. Since the Riemann--Lebesgue measure is $\sigma$-finite on $S$, it follows from \cite[Proposition 3.4.5]{Cohn} that $\Omega_{L^2}(S)$ admits a countable basis.
\end{Remark}
 
\begin{Remark} \label{rmk:subsurface}
Let $U \subseteq S$ be an open subset. Then $U$ is itself a Riemann surface  and we have a natural inclusion (extension by zero) of Hilbert spaces $\Omega_{L^2}(U) \xhookrightarrow{} \Omega_{L^2}(S)$.
\end{Remark}
\subsection{Convergence of measures and forms} \label{sec:measform}

We are mainly concerned with the measurable space consisting of a Riemann surface $S$ together with its Borel $\sigma$-algebra $\mathcal{B}$.

\medskip

A sequence of measures $\{\mu_n\}$ on a measurable space $(S, \mathcal{B})$ is said to {\em converge strongly} to a measure $\mu$ if we have ${\mu_n(A)} \rightarrow \mu(A)$ for every $A \in \mathcal{B}$. 

\medskip

Given a nonnegative $(1,1)$-form $\eta$ on a Riemann surface $S$, the map
\[
A\mapsto \int_A\eta
\] 
defines a measure on $S$, which we denote by $\mu_\eta$. Here, the set $A$ is considered together with the orientation inherited from the surface $S$.

\medskip

A sequence of nonnegative $(1,1)$-forms $\{\eta_n\}$ on $S$ {\em converges weakly} to a $(1,1)$-form $\eta$ if the sequence of associated measures $\{\mu_{\eta_n}\}$ converges strongly to the measure $\mu_\eta$. 

\medskip

Fix a finite analytic atlas $\{(U_1, z_1), \ldots , (U_k, z_k)\}$ for a {compact} Riemann surface $S$. We say a sequence of nonnegative $(1,1)$-forms $\{\eta_n\}$ on $S$ {\em converges uniformly} to a $(1,1)$-form $\eta$ if $\{\eta_n\}$ converges uniformly to $\eta$ on each coordinate chart $(U_i, z_i)$. If $z = x+iy$ is a local analytic coordinate in a domain $U \subset S$ and
\[ \eta_n = \rho_n \, dx \wedge dy = ({i}/{2}) \, \rho_n \, dz \wedge d\bar{z} \, ,\]
\[\eta = \rho \, dx \wedge dy = ({i}/{2}) \, \rho \, dz \wedge d\bar{z}\, ,\] 
we say $\{\eta_n\}$  converges uniformly to $\eta$ on $(U, z)$ if the sequence of real-valued nonnegative functions $\{\rho_n\}$ converges uniformly to $\rho$ on $U$.

\section{Canonical measures and forms}\label{sec:canonical_measure}

\subsection{Canonical measures}

Let $S$ be a Riemann surface, and let $\mathcal{H}_{L^2}(S) \subseteq \Omega_{L^2}(S)$ be the subspace consisting of harmonic $1$-forms: 
\begin{equation} \label{eq:harmonics}
\mathcal{H}_{L^2}(S) = \{ \omega \in \Omega_{L^2}(S) \colon \Delta(\omega)=0  \} \, ,
\end{equation}
where $\Delta$ denotes the usual {\em Hodge Laplacian} (also known as the {\em Laplace--de Rham} operator) on $1$-forms. It is self-adjoint with respect to the hermitian inner product in equation \eqref{eq:innerprod}: if $\omega_1, \omega_2 \in \Omega_{L^2}(S)$ are supported on the interior of $S$ we have
\begin{equation}\label{eq:selfadj}
\langle \Delta(\omega_1) , \omega_2 \rangle = \langle \omega_1, \Delta(\omega_2) \rangle \, 
\end{equation}
Because the operator $\Delta$ is elliptic, it follows from elliptic regularity (see, e.g., \cite[Chapter 6]{evans1998partial}) that harmonic $1$-forms are indeed smooth. Moreover, $\mathcal{H}_{L^2}(S)$ forms a Hilbert subspace of $\Omega_{L^2}(S)$. This follows from, e.g., the $L^2$ Hodge--de Rham decomposition (\cite[Theorem 1.57]{luck}).

\begin{Definition} \label{def:canmeas}
Let $S$ be a Riemann surface. Let $\{u_j\}_{j\in J}$ be an orthonormal basis for the Hilbert space $\mathcal{H}_{L^2}(S)$.
The {\em canonical measure} on $S$ is defined by
\[
\mu_{\can}^S(A)=\frac{1}{2} \sum_{j\in J} \int_A u_j\wedge \star \bar{u}_j
\]
for any Borel subset $A \subseteq S$. 
\end{Definition}
\begin{Remark} \phantomsection \label{rmk:canmeas}
\begin{itemize}\label{rmk:mucan}
\item[]
\item[(i)] If the surface $S$ is clear from the context, we will use $\mu_{\can}$ instead of $\mu_{\can}^S$.
\item[(ii)] The index set $J$ is countable (Remark~\ref{rmk:separable}). Therefore $\mu_{\can}$ is indeed a measure on $S$, as it is a countable sum of integrals of smooth $2$-forms. 
\item[(iii)] If one uses the opposite orientation on $S$, the forms $\star u_j$ become their negative and the measure $\mu_{\can}$ remains unchanged. Our computations and arguments about $\mu_{\can}$ can be generalized to the case where $S$ is not orientable.
\end{itemize}
\end{Remark}
\begin{Lemma}\phantomsection \label{lemcan} 
The following statements hold:
\begin{itemize}
  \item[(a)] The definition of $\mu_{\can}$ is independent of the choice of the orthonormal basis $\{u_j\}_{j\in J}$.
  \item[(b)] If $\phi: S'\rightarrow S$ is a Galois covering map, the canonical measure on $S'$ is the pullback of a measure on $S$, which we shall henceforth denote by $\mu_{\can, \phi}$.
   \item[(c)] If $S$ is compact then $\mu_{\can}(S)=\frac{1}{2}\dim_{\CC}\mathcal{H}_{L^2}(S) = g$.
   \item[(d)] The measure $\mu_{\can}$ depends only on the complex structure on $S$ (and not on the particular choice of the Riemannian metric).
   \item[(e)] $\mu_{\can}$ is invariant under conformal transformations.
  \end{itemize}
\end{Lemma}

\begin{proof} 
(a) For a Borel subset $A \subseteq S$, consider the nonnegative self-adjoint operator $F_A: \mathcal{H}_{L^2}(S)\rightarrow\mathcal{H}_{L^2}^{\vee}(S) \simeq \mathcal{H}_{L^2}(S)$ defined by $(F_A(\alpha))(\beta)=\int_A \alpha\wedge \star \bar{\beta}$. The trace of $F_A$, with respect to the orthonormal basis $\{u_j\}_{j\in J}$, is:
\[
{\rm Tr}(F_A)=\sum_{j\in J} \langle F_A(u_j),u_j\rangle= 2\mu_{\can}(A) \, .
\]
The result now follows from the independence of trace from the choice of basis.

(b) It follows from the definition that $\mu_{\can}$ is invariant under isometries.

(c) When $S$ is compact, we have $|J| = \dim_{\CC}\mathcal{H}_{L^2}(S) = 2g$.

(d) This follows from the fact that the Hodge star operator depends only on the complex structure (see \eqref{star_eq}).

(e) This is the immediate consequence of part (d).
\end{proof}

\subsection{Properties of canonical measures}

Some of the following properties of $\mu_{\can}^S$ follow from the known properties of canonical forms (see \S\ref{sec:arakelovrelation}), but we will give a self-contained treatment here.
We begin by giving an alternate formula for the canonical measure of {\em open sets} in terms of orthogonal projections.

\begin{Proposition}\label{prop:cm_projection_formula}
Let $S$ be a Riemann surface. Let $\pi_S \colon \Omega_{L^2}(S) \rightarrow \mathcal{H}_{L^2}(S)$ be the orthogonal projection, and let $A\subseteq S$ be any open subset. Then
\[\mu_{\can}(A)= \frac{1}{2}\sum_{k \in I}\langle \omega_k, \pi_S (\omega_k)\rangle \, ,\]
where $\{\omega_k\}_{k \in I}$ is an orthonormal basis for the Hilbert subspace $\Omega_{L^2}(A) \subseteq \Omega_{L^2}(S)$. 
\end{Proposition}
\begin{proof}
Let $\{u_j\}_{j \in J}$ be an orthonormal basis for $\mathcal{H}_{L^2}(S)$. Then  
\[
\begin{aligned}
\sum_{k \in I}\langle \omega_k, \pi_S (\omega_k)\rangle &=\sum_{k \in I} \langle \omega_k, \sum_{j\in J} \langle u_j, \omega_k\rangle u_j\rangle \\
&=\sum_{k \in I} \sum_{j\in J}  \langle u_j, \omega_k\rangle\langle \omega_k, u_j\rangle \\
& =\sum_{k \in I} \sum_{j \in J} |\langle u_j, \omega_k\rangle|^2 \\
&=\sum_{j \in J} \sum_{k \in I} |\langle u_j, \omega_k\rangle|^2 &\text{(Tonelli's theorem)}\\
& =\sum_{j \in J} \int_{A} u_j\wedge \star \bar{u}_j &\text{(Parseval's theorem)}\\
&=2\mu_{\can}(A) \, .
\end{aligned}
\]
\end{proof}

\begin{Proposition}\label{cor:cm_convergence}
Let $S$ be a Riemann surface.
\begin{itemize}
\item[(a)] Assume $U\subseteq S$ is an open subset. For any open subset $A\subseteq U$ we have
  \[\mu_{\can}^U(A)\geq \mu_{\can}^S(A)\, .\]  
\item[(b)] Let $U_1\subseteq U_2\subseteq \cdots \subseteq S$ be a sequence of open subsets with $S=\bigcup_j U_j$. Then, for any open subset $A\subseteq S$,
\[\lim_{j\rightarrow\infty} \mu_{\can}^{U_j}(A)= \mu_{\can}^S(A)\]

\end{itemize}
\end{Proposition}

\begin{proof}
By Remark~\ref{rmk:subsurface}, we may consider $\Omega_{L^2}(U)$ and $\Omega_{L^2}(U_j)$'s as Hilbert subspaces of $\Omega_{L^2}(S)$.
Let $\{\omega_k\}_{k \in I}$ be an orthonormal basis for $\Omega_{L^2}(A)$. By Proposition~\ref{prop:cm_projection_formula} we have 
\begin{equation} \label{eq:comp1}
\begin{aligned}
&2\mu_{\can}^U(A)= \sum_{k \in I}\langle \omega_k, \pi_U(\omega_k)\rangle
\, , \\
&2\mu_{\can}^{U_j}(A)= \sum_{k \in I}\langle \omega_k, \pi_{U_j}(\omega_k)\rangle \, , \\
&2\mu_{\can}^S(A)= \sum_{k \in I}\langle \omega_k, \pi_S(\omega_k)\rangle \, ,
\end{aligned}
\end{equation}
where $\pi_S \colon  \Omega_{L^2}(S) \rightarrow \mathcal{H}_{L^2}(S)$ and $\pi_U \colon \Omega_{L^2}(U) \rightarrow \mathcal{H}_{L^2}(U)$ and $\pi_{U_j} \colon \Omega_{L^2}(U_j) \rightarrow \mathcal{H}_{L^2}(U_j)$ denote the orthogonal projections.

(a) Because harmonicity is local, we have $\mathcal{H}_{L^2}(S) \subseteq \mathcal{H}_{L^2}(U)$. 

Let $\mathcal{H}'$ be the space consisting of $L^2$ $1$-forms on $S$ which are also harmonic on $U$ (and therefore smooth on $U$ by elliptic regularity). Note that $\mathcal{H}'$ is a Hilbert subspace of $\Omega_{L^2}(S)$. This is because $\mathcal{H}'$ is the orthogonal complement of the set $\{\Delta(\omega) \colon \omega \in \Omega_{L^2}(S) \text{ is supported on $U$}\}$ (see \eqref{eq:selfadj}), and weak harmonic forms are the same as strong harmonic forms. Let $\pi' \colon\Omega_{L^2}(S) \rightarrow \mathcal{H}'$ be the orthogonal projection. Then 
\begin{equation} \label{eq:comp2}
\langle \omega_k, \pi_U(\omega_k)\rangle=\langle \omega_k, \pi'(\omega_k)\rangle\, . 
\end{equation}

In other words, the orthogonal complement of $\mathcal{H}' \cap \Omega_{L^2}(U)$ in $\mathcal{H}'$ is the same as $\Omega_{L^2}(U)^\perp$ in  $\Omega_{L^2}(S)$. 
To see this, note that $\pi_U(\omega_k)$ is the closest point in $\mathcal{H}_{L^2}(U)$ to $\omega_k$, and $\pi'(\omega_k)$ is the closest point in $\mathcal{H}'$ to $\omega_k$. But, for any $\beta \in \mathcal{H}'$, we have $\|\beta - \omega_k\|^2 = \|\beta|_U-\omega_k\|^2+\|\beta|_{S\backslash U}\|^2$. Therefore $\pi'(\omega_k)$ must be the same as $\pi_U(\omega_k)$ on $U$, and must vanish otherwise.

We have 
\begin{equation} \label{eq:comp3}
\langle \omega_k, \pi'(\omega_k)\rangle\geq \langle \omega_k, \pi_S(\omega_k)\rangle
\end{equation}
because $\mathcal{H}_{L^2}(S)\subseteq \mathcal{H}'$. The result follows by putting together \eqref{eq:comp1}, \eqref{eq:comp2}, and \eqref{eq:comp3}.

\medskip

(b) Let $\mathcal{H}'_{j}$ be the space of $L^2$ $1$-forms on $S$ which are harmonic on $U_j$. Again, $\mathcal{H}'_{j}$ is a Hilbert subspace of $\Omega_{L^2}(S)$ and its elements are smooth on $U_j$. Since $\bigcap_j\mathcal{H}'_{j}=\mathcal{H}_{L^2}(S)$ we have
\begin{equation} \label{eq:comp4}
\lim_{j \rightarrow\infty}\langle \omega_k,\pi_{U_j}(\omega_k)\rangle=\lim_{j\rightarrow\infty}\langle \omega_k,\pi'_j(\omega_k)\rangle=\langle \omega_k,\pi_S (\omega_k)\rangle \, ,
\end{equation}
where $\pi'_j \colon\Omega_{L^2}(S) \rightarrow \mathcal{H}'_j$ is the orthogonal projection.
 The last equality is a consequence of the following fact: for any Hilbert space and a
decreasing sequence of its Hilbert subspaces, the sequence of orthogonal projections onto subspaces converge pointwise to the orthogonal projection onto the intersection of those subspaces (this is an easy consequence of the Gram--Schmidt process and Parseval's theorem). 
The result follows from \eqref{eq:comp1} and \eqref{eq:comp4}.
\end{proof}

\begin{Proposition} \label{prop:radon}
Let $S$ be a Riemann surface.
\begin{itemize}
\item[(a)] $\mu_{\can}$ is absolutely continuous with respect to the Riemann--Lebesgue measure. 
\item[(b)] $\mu_{\can}$ is a Radon measure.
\end{itemize}
\end{Proposition}
\begin{proof}
(a) We show that the measure $\mu_{\can}$ is absolutely continuous with respect to any Riemann-Lebesgue measure computed with any smooth Riemannian metric in the conformal class. It does not matter which Riemannian metric is used, because if two different Riemannian metrics result in two measures $\mu_1$ and $\mu_2$ then there is a smooth positive function $f$ so that $\mu_2=f\mu_1$, hence $\mu_2$ is absolutely continuous with respect to $\mu_1$ and vice versa.

Let $D$ be the Poincar\'e unit disk. An orthogonal basis of $\mathcal{H}_{L^2}(D)$ is 
\[\{z^ndz ,  \bar{z}^nd\bar{z}\colon n \in \ZZ^{\geq 0}\} \, .\] 
For any Borel subset $A\subseteq D$ we have, by Definition~\ref{def:canmeas},
\[\mu_{\can}^D(A)={i\over 2} \int_A \left( \sum_n{n+1\over\pi} |z|^{2n} dz \wedge d\bar{z} - \sum_n{n+1\over\pi} |z|^{2n} d\bar{z} \wedge dz \right)\]
\[=\int_A \frac{i}{\pi}\sum_n{(n+1)}|z|^{2n}dz \wedge d\bar{z}=\int_A{\frac{i}{\pi(1-|z|^2)^2}}dz \wedge d\bar{z}\]
Therefore $\mu_{\can}^D$ is absolutely continuous with respect to the Riemann--Lebesgue measure on the disk. 
For a general surface $S$, around any point there is a conformally embedded open disk. Therefore the result follows from Proposition~\ref{cor:cm_convergence}~(a). 

(b) Since $S$ is a second countable, locally compact Hausdorff space, we only need to show that $\mu_{\can}$ is finite on compact sets (see, e.g., \cite[Theorem 7.8]{Folland}).
The Riemann-Lebesgue measure is locally finite and, by part (a), $\mu_{\can}$ is absolutely continuous with respect to the Riemann-Lebesgue measure. Furthermore, from the computation in part (a), we see that the Radon--Nikodym derivative of $\mu_{\can}$ with respect to the Riemann-Lebesgue measure is bounded from above by a locally bounded function. This implies that the $\mu_{\can}$ is bounded on compact sets.
\end{proof}

\subsection{Canonical (Arakelov) forms} \label{sec:arakelovrelation}

\subsubsection{Compact Riemann surfaces}
Let $S$ be a compact Riemann surface of genus $g \geq 1$. Fix an orthonormal (with respect to the Hodge inner product) basis $\{\varphi_1, \cdots , \varphi_g\}$ for the vector space of (global) holomorphic $1$-forms (differentials of the first kind) $\Omega^{1,0}(S)$. 
Following \cite{Arakelov, Faltings} (see also \cite[II,\S2]{LangAr} or \cite[\S3]{deJong}), the {\em canonical (Arakelov) $(1,1)$-form} of $S$ is defined to be 
 \[\eta_{\can} = \frac{i}{2} \sum_{k=1}^g{\varphi_k \wedge \bar{\varphi}_k} \, .\]

It is easy to check (using Riemann--Roch theorem) that $\eta_{\can}$ is indeed a {\em volume form}, and \[\int_S \eta_{\can} = g \, .\]

\begin{Remark} \phantomsection
\begin{itemize}
\item[]
\item[(i)] On the Jacobian $J(S)$ of $S$ there exists a canonical translation-invariant $(1,1)$-form $\eta_J$ obtained by identifying $J(S)$ with $\Omega^{1,0}(S)^*/H_1(S, \ZZ)$. One can easily check that $\eta_{\can}$ is obtained by pulling back $\eta_J$ along the Abel--Jacobi map.
\item[(ii)] In Arakelov geometry, it is customary to study $\eta_{\can} /g$ instead of $\eta_{\can}$. This normalization is not suitable in our context, as we will also deal with non-compact Riemann surfaces.
\end{itemize}
\end{Remark}

\begin{Remark} \phantomsection \label{rmk:metricsVSforms}
\begin{itemize}
\item[]
\item[(i)] It is known that $\eta_{\can}$ can be obtained as the Chern form of the canonical line bundle $K_S$ equipped with the {\em Arakelov metric} \cite[\S4]{Arakelov}. For a description of the Arakelov metric in terms of {\em Arakelov Green's function} see \cite[II,\S2]{LangAr} or \cite[\S3]{deJong}. 
\item[(ii)] The {\em canonical metric} on $S$ is, by definition, the hermitian metric on the holomorphic tangent bundle $T_S$ defined by $ds^2_{\can} = \sum_{k=1}^g \varphi_k \otimes \bar{\varphi}_k$ (for the orthonormal basis $\{\varphi_1, \cdots , \varphi_g\}$ as above). The canonical $(1,1)$-form $\eta_{\can}$ is the {\em associated $(1, 1)$-form} of $ds^2_{\can}$, i.e. $\eta_{\can} = -(1/2){\rm Im}\,{ds^2_{\can}}$. It is well-known (see, e.g., \cite[pp 28--29]{GriffHar}) that $\eta_{\can}$ uniquely determines $ds^2_{\can}$.
\item[(iii)] Let the $(1,1)$-form $\eta_B$ denote the Chern form of the pair $(T_S, ds^2_{\can})$. Let $ds^2_B$ be the corresponding hermitian metric on the holomorphic tangent bundle $T_S$. More precisely, $\eta_{B} = -(1/2){\rm Im}\, {ds^2_{B}}$ and $\eta_{B}$ is the {\em associated $(1, 1)$-form} of $ds^2_{B}$. It can be checked that $ds^2_{B}$ coincides with the pullback Fubini--Study metric under the canonical mapping $S \rightarrow \mathbb{P}^{g-1}$.
\item[(iv)] In the literature, the term {\em `Bergman metric'} sometimes refers to $ds^2_{\can}$ (see, e.g., \cite{Neeman, mcmullen2013entropy}), and sometimes refers to $ds^2_{B}$ (see, e.g., \cite{mumford, rhodes1993sequences}).
\end{itemize}
\end{Remark}

\subsubsection{General Riemann surfaces}
Let $S$ be a (possibly non-compact) Riemann surface. Fix an orthonormal basis (with respect to the Hodge inner product) $\{\varphi_k\}$ for the Hilbert space of holomorphic $1$-forms.
The canonical (Arakelov) $(1,1)$-form of $S$ is defined to be the nonnegative $(1,1)$-form
\[
\eta_{\can} = \frac{i}{2} \sum_k \varphi_k \wedge \bar{\varphi}_k \, .
\]
Recall (see \S\ref{sec:measform}) the map $A \mapsto \int_A{\eta_{\can}}$ defines a measure $\mu_{\eta_{\can}}$ on $S$.
\begin{Lemma}\label{lem:measvsform}
We have
$\mu_{\can}^S = \mu_{\eta_{\can}}$.
\end{Lemma}
\begin{proof}
Let $\{\varphi_k\}$ be an orthonormal basis (with respect to the Hodge inner product) for the vector space of holomorphic $1$-forms. An orthonormal basis (with respect to the product given in \eqref{eq:innerprod}) for $\mathcal{H}_{L^2}(S)$ is $\{\frac{1}{\sqrt{2}}\varphi_k , \frac{1}{\sqrt{2}}\bar{\varphi}_k\}$. For a Borel set $A \subseteq S$, by Definition~\ref{def:canmeas}, we have
\[
\begin{aligned}
\mu_{\can}^S(A) &= \frac{1}{2} \left(\sum_k \int_A{\frac{1}{\sqrt{2}}\varphi_k \wedge \star\frac{1}{\sqrt{2}}\bar{\varphi}_k} + \sum_k \int_A{\frac{1}{\sqrt{2}}\bar{\varphi}_k \wedge \star\frac{1}{\sqrt{2}}{\varphi}_k} \right)\\
&= \frac{1}{2} \left(\sum_k \int_A{\frac{i}{2} \varphi_k \wedge \bar{\varphi}_k} - \sum_k \int_A{\frac{i}{2}\bar{\varphi}_k \wedge {\varphi}_k} \right)\\
&= \int_A{\eta_{\can}} \\
&= \mu_{\eta_{\can}}(A) \, .
\end{aligned}
\]

Note that $\varphi \wedge \star \bar{\varphi} = i \varphi \wedge \bar{\varphi}$ and $\bar{\varphi} \wedge \star {\varphi} = -i \bar{\varphi} \wedge {\varphi}$ for any $\varphi \in \Omega_{L^2}^{1,0}(S)$.
\end{proof}

There is also a nice local description for $\eta_{\can}$ which we now describe (see \cite[Appendix]{mcmullen2013entropy}). Consider any holomorphic local chart $(U_P, z\colon U_P \xrightarrow{\sim}D_\epsilon)$, sending a neighborhood $U_P$ of some point $P \in S$ to an $\epsilon$-disk $D_\epsilon \subset \CC$ around $0 \in \CC$. The canonical $(1,1)$-form $\eta_{\can}$ can be expressed on $U_P$ in terms of the coordinate $z$ as
  \[({i}/{2}) \, \rho \, dz\wedge d\bar{z}\] 
for some real-valued nonnegative function $\rho \colon D_\epsilon \rightarrow \RR$. If $\{\varphi_k\}$ is an orthonormal basis of the space of holomorphic $1$-forms on $S$, and $\varphi_k = a_k(z)dz$ in terms of the coordinate $z$, then $\rho(z) = \sum_k |a_k(z)|^2$.
  \begin{Lemma}\label{lem1} 
  We have 
    \[\rho(z)=\max \{|a(z)|^2 \colon \varphi \text{ is holomorphic on }S, (\varphi, \varphi)=1, \varphi|_{U_P}=a(z)dz\} \, .\]
  \end{Lemma}
  \begin{proof}
    Let $\{\varphi_k\}$ be an orthonormal basis of the space of holomorphic $1$-forms on $S$, and let $\varphi_k = a_k(z)dz$ in terms of the coordinate $z$. We have
 \[   
 \begin{aligned}
 & \max \{|a(z)|^2 \colon \varphi \text{ is holomorphic on }S, (\varphi, \varphi)=1, \varphi|_{U_P}=a(z)dz\}  \\
&= \max \{|\sum_k{c_k a_k(z)}|^2 \colon \sum_j |c_j|^2=1, c_j \in \CC \}\\
&= {\sum_k |a_k(z)|^2}={\rho(z)} \, .
 \end{aligned}
    \]
    The second equality is by Cauchy--Schwarz.  
    \end{proof}

\section{A variant of L\"uck's approximation}\label{sec:l2approx}

\subsection{Hilbert $\G$-modules and $\G$-dimensions}\label{sec:Gdim}

We quickly review the von Neumann algebras and dimensions that appear in our context. See \cite{luck} for proofs and a more thorough treatment. 

For any complex Hilbert space $\H$, let $\B(\H)$ denote the algebra of all bounded linear operators on $\H$, and let $\B(\H)^+ = \{A \in \B(\H) \colon \langle Ax , x \rangle \in \RR^{\geq 0} \, , \, \forall x \in \H \}$.
By a {\em Hilbert $\G$-module} we mean a Hilbert space $\H$ together with a (left) unitary action of the discrete group $\G$. 
A {\em free Hilbert $\G$-module} is a Hilbert $\G$-module which is unitarily isomorphic to $\ell^2(\G) \otimes \H$, where $\H$ is a Hilbert space with the trivial $\G$-action and the action of $\G$ on $\ell^2(\G)$ is by left translations. 
Let $\{u_\alpha\}_{\alpha \in J}$ be an orthonormal basis for $\H$. Then we have an orthogonal decomposition
\[
\ell^2(\G) \otimes \H = \bigoplus_{\alpha \in J} \ell^2(\G)^{(\alpha)} = \{\sum_{\alpha \in J}{f_{\alpha} \otimes u_{\alpha}} \colon f_{\alpha} \in \ell^2(\G) \, , \,  \sum_{\alpha \in J} \|f_{\alpha}\|^2 < +\infty \} \, , 
\]
where $\ell^2(\G)^{(\alpha)} = \ell^2(\G) \otimes u_{\alpha}$ is a copy of $\ell^2(\G)$. 

For each $g \in \G$, we have the left and right translation operators $L_g , R_g \in \B(\ell^2(\G))$ defined by $L_g(f)(h) = f(g^{-1}h)$, $R_g(f)(h) = f(hg)$ for $h \in \G$. We are interested in the von Neumann algebra $\M_r(\G) \otimes \B(\H)$ on $\ell^2(\G) \otimes \H$, where $\M_r = \M_r(\G)$ is the von Neumann algebra generated by $\{R_g \colon g \in \G \} \subseteq \B(\ell^2(\G))$. Alternatively $\M_r$ is the algebra of $\G$-equivariant (i.e. those commuting with all operators $L_g$) bounded operators on $\ell^2(\G)$.

\begin{Definition} \label{def:trG}
Let $\{u_\alpha\}_{\alpha \in J}$ be an orthonormal basis for the Hilbert space $\H$.
For every $A \in (\M_r(\G) \otimes \B(\H))^+$, define 
\[
\Tr_{\G}(A)  = \sum_{\alpha \in J} \langle A(\delta_h \otimes u_{\alpha}) , \delta_h \otimes u_{\alpha} \rangle \, ,
\]
for a fixed $h \in \G$. Here $\delta_h$ denotes the indicator function of $h$. This is independent of the choice of $h$, so one usually picks $h = {\rm id}$, the group identity. It can be checked that this is a ``trace function'' in the sense of von Neumann algebras.
\end{Definition}

A {\em projective Hilbert $\G$-module} is a Hilbert $\G$-module $\V$ which is unitarily isomorphic to a closed submodule of a free Hilbert $\G$-module, i.e. a closed $\G$-invariant subspace in {\em some} $\ell^2(\G) \otimes \H$.
Note that the embedding of $\V$ into $\ell^2(\G) \otimes \H$ is {\em not} part of the structure; only its existence is required.
Fix such an embedding. Let $P_{\V}$ denote the orthogonal projection from $\ell^2(\G) \otimes \H$ onto $\V$. Then $P_{\V} \in \M_r(\G) \otimes \B(\H)$ because it commutes with all $L_g \otimes I$.
The $\G$-dimension of $\V$ is defined as
\begin{equation}\label{eq:gdim}
\dim_{\G}(\V) = \Tr_{\G}(P_{\V}) \, .
\end{equation}

An elementary fact is that $\dim_{\G}(\V)$ does {\em not} depend on the choice of the embedding of $\V$ into a free Hilbert $\G$-module; it is a well-defined invariant of $\V$. 

\begin{Remark} \label{rmk:dimGprops} 
The $\G$-dimensions satisfy the following expected properties: 
\begin{itemize}
\item[(i)] $\dim_{\G}(\V) = 0$ if and only if $\V = \{0\}$.
\item[(ii)] $\dim_{\G}(\ell^2(\G)) = 1$.
\item[(iii)] $\dim_{\G}(\ell^2(\G) \otimes H) = \dim_{\CC}(H)$.
\item[(iv)] $\dim_{\G}(\V_1\oplus \V_2)=\dim_{\G}(\V_1) + \dim_{\G}(\V_2)$.
\item[(v)] $ \V_1 \subseteq \V_2$ implies $\dim_{\G}(\V_1) \leq \dim_{\G}(\V_2)$. Equality holds if and only if $\V_1 = \V_2$.
\item[(vi)] If $0\rightarrow \mathrm{U} \rightarrow \V \rightarrow \mathrm{W} \rightarrow 0$ is a short {\em weakly exact sequence} of projective Hilbert $\G$-modules, then $\dim_{\G}(\V) = \dim_{\G}(\mathrm{U})+\dim_{\G}(\mathrm{W})$.
\item[(vii)] If $\V$ and ${\rm W}$ are {\em weakly isomorphic}, then $\dim_{G}(\V) = \dim_{G}({\rm W})$.
\end{itemize}
A sequence of $\mathrm{U} \xrightarrow{i} \V \xrightarrow{p} \mathrm{W}$ of projective Hilbert $\G$-modules is called {\em weakly exact} at $\V$ if ${\rm Kernel}(p) =  \cl({\rm Image}(i))$. A map of projective Hilbert $\G$-modules $\V \rightarrow {\rm W}$ is a {\em weak isomorphism} if it is injective and has dense image.
\end{Remark}

\subsection{$L^2$ Hodge--de Rham theorem}

Let $X' \rightarrow X$ be a Galois covering of a finite CW complex $X$, with $\G$ as the group of deck transformations. 

Let $C^*_{L^2}(X')$ denote its cellular $L^2$-cochain complex:
\[
C^*_{L^2}(X') = {\rm Hom}_{\ZZ \G}\left(\ell^2(\G) \otimes_{\ZZ \G} C_*(X') , \ell^2(\G) \right) \, ,
\]
where $C_*(X')$ is the usual cellular chain complex, considered as a $\ZZ \G$-module.
Fixing a cellular basis for $C_p(X')$, one obtains an explicit isomorphism
\[
C^p_{L^2}(X') \simeq \left(\ell^2(\G)\right)^{n_p}
\]
for some integer $n_p$. Therefore $C^p_{L^2}(X')$ has the structure of a projective Hilbert $\G$-module.  Let $d_{L^2}^p \colon C^p_{L^2}(X') \rightarrow C^{p+1}_{L^2}(X')$ denote the induced $L^2$-coboundary map.
The (reduced) {\em $p$-th $L^2$-cohomology} of the pair $(X',\G)$ is defined by
\[
H^p_{L^2}(X'/ \G) = {\rm Ker}(d_{L^2}^p) / {\rm cl} \left({\rm Image}(d_{L^2}^{p-1})\right) \, .
\]
Note that, since we divide by the {\em closure} of the image, the resulting $H^p_{L^2}(X'/ \G)$ inherits the structure of a Hilbert space. It is, moreover, a projective Hilbert $\G$-module because $H^p_{L^2}(X'/ \G) = {\rm Ker}(d_{L^2}^p) \cap {\rm Image}(d_{L^2}^{p-1})^{\perp}$. Therefore it makes sense to define the {\em $p$-th $L^2$-Betti number} of the pair $(X', \G)$ by
\[
b^p_{L^2}(X'/ \G) = \dim_{\G} H^p_{L^2}(X'/ \G) \, .
\] 
The following $L^2$ version of the Hodge-de Rham theorem is proved in \cite{Dodziuk} (see also \cite[Theorem 1.59]{luck}).

\begin{Theorem}[$L^2$ Hodge--de Rham theorem] \label{thm:HodgedeRham}
Let $M' \rightarrow M$ be a Galois covering of a compact Riemannian manifold  $M$, with $\G$ as the group of deck transformations. Assume further that $M'$ has no boundary. Let $X'$ be an equivariant smooth triangulation of $M'$. Then there is a canonical isomorphism (as projective Hilbert $\G$-modules) between the space of $L^2$ harmonic smooth $p$-forms and the $p$-th $L^2$-cohomology of the pair $(X', \G)$: 
\[\mathcal{H}^p_{L^2}(M') \simeq H^p_{L^2}(X'/ \G) \, .\]
Furthermore, $b^p_{L^2}(X'/ \G) = \dim_{\G} H^p_{L^2}(X'/ \G) = \dim_{\G}\mathcal{H}^p_{L^2}(M')$ is finite.
\end{Theorem}

\begin{Remark} \phantomsection \label{rmk:L2betti}
\begin{itemize}
\item[]
\item[(i)] We only use this theorem when $M' = S'$ and $M = S$ are Riemann surfaces and $p=1$. In this case, the space of $L^2$ harmonic smooth $1$-forms $\mathcal{H}^1_{L^2}(M)$ is denoted by $\mathcal{H}_{L^2}(S)$ in \eqref{eq:harmonics}. General $\mathcal{H}^p_{L^2}(M)$ spaces are defined analogously in higher dimensions.
\item[(ii)] It is well known that if $M$ is a Riemann surface of genus $g \geq 1$ we have $\dim_{\G}\mathcal{H}^1_{L^2}(M') = 2g-2$. This follows from \cite[Theorem 1.35]{luck} (see also \cite[Example 1.36]{luck}). Alternatively, one can deduce this from our Thoerem~\ref{thm:luck_l2dim}.
\end{itemize}
\end{Remark}

\subsection{Approximation theorem} 

We will prove a variant of L\"uck's approximation theorem in \cite{Luck2}. 
Let $\G$ be a discrete group as before. Let $f \colon (\mathbb{Z}\G)^n \rightarrow (\mathbb{Z}\G)^m$ be a $\mathbb{Z}\G$-module homomorphism. After tensoring with $\CC$ and completion, we obtain an induced map
\[
f^{(2)}: (l^2(\G))^n\rightarrow(l^2(\G))^m \, .
\]
For any finite index normal subgroup $\G_k \trianglelefteq \G$, we also have a map
\[
f_k: (\mathbb{C}[\G /\G_k])^n\rightarrow (\mathbb{C}[\G/\G_k])^m
\]
induced by the quotient maps and tensoring with $\CC$. 
Concretely, if $F$ denotes the standard matrix of $f$ with respect to the standard bases for $(\mathbb{Z}\G)^n$ and $(\mathbb{Z}\G)^m$, then $F$ also represents the standard matrix of $f^{(2)}$ and of $f_k$. 
In this situation, we have the following beautiful theorem of L\"uck (see \cite[Theorem 2.3]{Luck2}):
\begin{Theorem} \label{thm:luck_kerdim}
Let $\{\G_k\}$ be a descending sequence of finite index normal subgroups of $\G$ such that $\bigcap_k\G_k=\{ {\rm id} \}$.  
Then, $$\dim_\G\Ker(f^{(2)})=\lim_k{\dim_{\CC}\Ker(f_k)\over[\G:\G_k]}\, .$$ 
\end{Theorem}

The following result is an appropriate modification of \cite[Theorem 0.1]{Luck2}, needed for our application. 

\begin{Theorem} \label{thm:luck_l2dim}
Let $X$ be a finite connected CW complex. Let $X' \rightarrow X$ be an infinite Galois covering. Let $\{X_k \rightarrow X\}_k$ be an ascending sequence of finite Galois coverings converging to $X'$, in the sense that the equality 
\[ \bigcap_k\pi_1(X_k)=\pi_1(X')\]
holds in $\pi_1(X)$. Let
\begin{itemize}
\item $\G = \pi_1(X)/\pi_1(X')$ denote the deck transformation group of the covering $X' \rightarrow X$,
\item $r_k = [\pi_1(X):\pi_1(X_k)]$ denote the degree of the covering $X_k \rightarrow X$, 
\item $H^p(X_k, \CC)$ denote the (ordinary) cohomology of $X_k$.
\end{itemize}
Then
\[
\dim_\G H_{L^2}^p(X'/ \G)=\lim_k \frac{\dim_{\CC} H^p(X_k, \CC)}{r_k} \, . 
\]
\end{Theorem}
Our proof is a modification of L\"uck's proof of \cite[Theorem 0.1]{Luck2}. There, the proof is given for the special case where $X'$ is the universal cover of $X$.

\begin{proof}
Let $f$ be the (ordinary) coboundary map $d^p \colon C^p(X') \rightarrow C^{p+1}(X')$ on $X'$. Then $f^{(2)}$ is the $L^2$-coboundary
  map $d^p_{L^2} \colon C_{L^2}^p(X') \rightarrow C_{L^2}^{p+1}(X')$ on $X'$, and $f_k$ will be the (ordinary) coboundary map $d^p_k \colon C^p(X_k) \rightarrow C^{p+1}(X_k)$ on $X_k$.

   Let $\G_k=\pi_1(X_k)/\pi_1(X')$. Then $(\G_k)_{k \in \mathbb{N}}$ is a descending sequence of subgroups of $\G = \pi_1(X)/\pi_1(X')$. Moreover $\bigcap_k\pi_1(X_k)=\pi_1(X')$
  implies that $\bigcap_k\G_k=\{ {\rm id} \}$. Clearly, $[\G: \G_k] = r_k$ and, by Theorem~\ref{thm:luck_kerdim}, we obtain: 
\begin{equation} \label{eqn:kerdim}
\dim_\G\Ker (d^p_{L^2})=\lim_k{\dim_{\CC}\Ker (d^p_k)\over r_k}\, .
\end{equation}

Let $n_p$ be the number of $p$-cells in $X$. Then 
\begin{itemize}
\item $C^p_{L^2}(X')$ is isomorphic to $(l^2(\G))^{n_p}$. By Remark~\ref{rmk:dimGprops} (ii) and (iv) we have:
\[\dim_{\G}C^p_{L^2}(X') = n_p \, .\]
\item $C^p(X_k)$ is a free $\CC[\G/\G_k]$-module of rank $n_p$. Hence \[\dim_{\CC} C^p(X_k) = n_p[\G:\G_k]=n_pr_k \, .\] 
\end{itemize}
We have 
\begin{equation}\label{eq:dimHp}
\begin{aligned}
\dim_{\CC} H^p(X_k, \CC) &=\dim_{\CC}\Ker (d^p_k)-\dim_{\CC}{\rm Image} (d^{p-1}_k)\\
&=\dim_{\CC} \Ker (d^p_k) - \left(n_{p-1}r_k-\dim_{\CC} \Ker (d^{p-1}_k) \right)\, .
\end{aligned}
\end{equation}
The second equality in \eqref{eq:dimHp} is by the (usual) rank--nullity theorem.
On the other hand,
\begin{equation}\label{eq:dimHpL2}
\begin{aligned}
\dim_\G H^p_{L^2}(X'/ \G)&=\dim_\G\Ker (d^p_{L^2})-\dim_\G \cl\left({\rm Image} (d^{p-1}_{L^2})\right)\\
&=\dim_\G\Ker (d^p_{L^2}) - \left(n_{p-1}-\dim_\G\Ker (d^{p-1}_{L^2})\right) \ .
\end{aligned}
\end{equation}
Both equalities in \eqref{eq:dimHpL2} follow from Remark~\ref{rmk:dimGprops} (vi), applied to the following short weakly exact sequences: 
\[
0 \rightarrow \cl\left({\rm Image} (d^{p-1}_{L^2})\right) \rightarrow \Ker (d^p_{L^2}) \rightarrow H^p_{L^2}(X'/ \G) \rightarrow 0 \, ,
\]
\[
0 \rightarrow \Ker (d^{p-1}_{L^2}) \rightarrow  C^{p-1}_{L^2}(X') \rightarrow \cl\left({\rm Image} (d^{p-1}_{L^2})\right) \rightarrow 0 \, .
\]

The result follows by putting together \eqref{eqn:kerdim}, \eqref{eq:dimHp}, and \eqref{eq:dimHpL2}.
\end{proof}

\section{A generalized Kazhdan's theorem}\label{sec:proof_Kazhdahn_surface}

\subsection{A Gauss--Bonnet type theorem}
We are now ready to state and prove the following Gauss--Bonnet type theorem.
\begin{Theorem} \label{thm:GaussBonnet} 
Let $\phi:S' \to S$ be an infinite Galois covering of a compact Riemann surface $S$, with $\G = \pi_1(S)/\pi_1(S')$ as the group of deck transformations. Then 
\[
\mu_{\can, \phi} (S)=\frac{1}{2}\dim_\G H^1_{L^2}(S'/ \G) \, .
\] 
\end{Theorem}

\begin{Remark}\label{rmk:L2betti2}
Recall $\dim_\G H^1_{L^2}(S'/ \G) = -\chi(S) = 2g-2$ by Remark~\ref{rmk:L2betti}~(ii).
\end{Remark}
\begin{proof}
Let $\pi_{S'} \colon \Omega_{L^2}(S') \rightarrow \mathcal{H}_{L^2}(S')$ denote the orthogonal projection, and let $F$ be a fundamental domain of the $\G$-action on $S'$. Let
$\{\epsilon_j\}_{j\in J}$ be an orthonormal basis for the space $\Omega_{L^2}(F) \simeq \Omega_{L^2}(S)$, embedded
into $\Omega_{L^2}(S')$ as a subspace by extending each $1$-form in
$\Omega_{L^2}(F)$ to a $1$-form in $\Omega_{L^2}(S')$ which vanishes outside $F$. We have, by Definition~\ref{def:trG}, 
\[
\Tr_\G (\pi_{S'}) = \sum_j \langle \epsilon_j, \pi_{S'}(\epsilon_j)\rangle \, ,\] 
which, by Proposition \ref{prop:cm_projection_formula}, equals $2 \mu_{\can}^{S'}(F)=2\mu_{\can, \phi}(S)$. The fact that 
\[\Tr_\G (\pi_{S'})=\dim_\G H^1_{L^2}(S'/ \G)\]
follows from Theorem~\ref{thm:HodgedeRham} and the definition of $\G$-dimension \eqref{eq:gdim}.
\end{proof}

\subsection{Convergence of canonical measures}
\begin{Theorem}\label{thm:Kazhdan_surface}
Let $S$ be a compact Riemann surface. Let $\phi: S'\rightarrow S$ be an
infinite Galois cover with deck transformation group $\G$. Let
$\{\phi_n \colon S_n\rightarrow S\}$ be a tower of $d_n$-fold finite Galois covers
between $S'$ and $S$ such that 
\[\bigcap_n \pi_1(S_n)=\pi_1(S')\, .\] 
Then $\mu_{\can, \phi_n}$ converges strongly to $\mu_{\can , \phi}$. 
\end{Theorem}

\begin{proof} 
Since every compact Riemann surface can be triangulated (i.e is homeomorphic to a simplicial complex), we have:
\begin{equation}\label{eq:conv}
\begin{aligned}
\lim_{n\rightarrow\infty}\mu_{\can, \phi_n}(S)&=\lim_{n\rightarrow\infty}{\mu^{S_n}_{\can}(S_n)\over d_n} \\
&=\lim_{n\rightarrow\infty}{\frac{\dim_{\CC} H^1(S_n, \CC)}{2d_n} }&\text{(Lemma~\ref{lemcan} (c))}\\
&=\frac{1}{2}\dim_\G H^1_{L^2}(S'/ \G)&\text{(Theorem \ref{thm:luck_l2dim})}\\
&= \mu_{\can, \phi}(S)\, . &\text{(Theorem~\ref{thm:GaussBonnet})}
\end{aligned}
\end{equation}

\noindent{\bf Claim.}
For any Borel measurable subset $A$ of $S$, we have $\lim\sup_{n\rightarrow\infty}\mu_{\can, \phi_n} (A) \leq\mu_{\can, \phi} (A)$. 

Before proving the claim, we note that the result follows from this claim; together with \eqref{eq:conv} it follows that for any Borel subset $U \subseteq S$ we have 
\[\lim_{n\rightarrow\infty}\mu_{\can, \phi_n}(U)=\mu_{\can, \phi}(U) \, .\] 
This is because $\lambda (\cdot) = \mu_{\can, \phi} (\cdot) - \lim\sup_{n\rightarrow\infty}\mu_{\can, \phi_n} (\cdot)$ is nonnegative for any Borel subsets $A\subseteq S$, and $\lambda(S)\geq \lambda(A)+\lambda(S\backslash A)$, as $\lim\sup$ preserves subadditivity.

\medskip

\noindent{\em Proof of the Claim.} Since all measures $\mu_{can, \phi_n}$ are Radon (by Proposition~\ref{prop:radon}~(b)) and $S$ is compact, one only needs to consider open subsets of $S$. Let $U \subseteq S$ be an open subset, let $U'$ be a lift of $U$ to $S'$, and let $U_n$ be a lift of $U$ to $S_n$. Let $V_j$ be an increasing, exhausting sequence of bounded open subsets on $S'$ that
contains $U'$. Then it follows from the condition $\bigcap_n \pi_1(S_n)=\pi_1(S')$ that each $V_j$ is embedded in all but finitely many of $S_n$'s.

By Proposition \ref{cor:cm_convergence} (a) we know 
$\lim\sup_{n\rightarrow\infty}\mu_{\can, \phi_n}|_U$ is bounded above by the canonical measure on $V_j$ restricted to $U'$. The latter converges to $\mu^{S'}_{\can}(U')=\mu_{\can, \phi}(U)$ by Proposition \ref{cor:cm_convergence} (b).
\end{proof}

\subsection{Convergence of canonical forms}
Theorem~\ref{thm:Kazhdan_surface} together with Lemma~\ref{lem:measvsform} give the {\em weak convergence of forms}. 
Our final goal is to enhance this convergence into a uniform convergence statement.

\begin{Theorem}\label{thm:arakelov conv}
Let $S$ be a compact Riemann surface. Let $\phi: S'\rightarrow S$ be an infinite Galois cover with deck transformation group $\G$. Let $\{\phi_n: S_n \rightarrow S\}$ be a tower of $d_n$-fold finite Galois covers between $S'$ and $S$ such that 
\[\bigcap_n \pi_1(S_n)=\pi_1(S')\, .\] 
Then the pushforward of the canonical $(1,1)$-forms on $S_n$ converge uniformly to the pushforward of the canonical $(1,1)$-form on $S'$.
\end{Theorem}

\begin{proof}
  Because $S$ is compact we only need to show the uniform convergence on an open disk around every point. Consider any holomorphic local chart $z$, sending a neighborhood $U_P$ of some point $P \in S$ to an $\epsilon$-disk $D_\epsilon \subset \CC$ around $0 \in \CC$.
  Lift the coordinate chart as well as $U_P$ to $S_n$ and to $S'$. Let the canonical $(1,1)$-form on $S_n$ and on $S'$ on that chart be 
  \[({i}/{2}) \, \rho_n \, dz\wedge d\bar{z} \quad \text{and} \quad  ({i}/{2}) \,\rho_{\infty} \, dz\wedge {d\bar{z}}\] 
  respectively, for some real nonnegative real-valued functions $\rho_n$ and $\rho_\infty$ on $D_\epsilon$.\\ 

By Lemma~\ref{lem1} we know:
    \[\rho_n(z)=\max \{|a(z)|^2 \colon \|\varphi\|=1, \varphi\text{ is holomorphic on }S_n, \varphi|_{U_P}=a(z)dz\} \, ,\]
    \[\rho_\infty(z)=\max \{|a(z)|^2 \colon \|\varphi\|=1, \varphi\text{ is holomorphic on }S', \varphi|_{U_P}=a(z)dz\} \, .\]

\medskip

\noindent {\bf Claim.}
  Both $\rho_n$ and $\rho_\infty$, restricted to the $\epsilon/2$ disk centered at $0$ are {\em uniformly Lipschitz}. More precisely, there exists $L >0$ (depending on $\epsilon$, but independent of $n$) such that
    \[
    |\rho_n(z_1) -\rho_n(z_2)| \leq L |z_1 - z_2| \quad{and} \quad     |\rho_\infty(z_1) -\rho_\infty(z_2)| \leq L |z_1 - z_2|
    \]
    for $|z_1|, |z_2|<\epsilon/2$.
 
 \medskip
 
\noindent {\em Proof of the claim.}
 By Lemma~\ref{lem1}, there is some holomorphic $1$-form on $S_n$ with norm $1$, locally of the form $a_n(z) dz$, such that $|a_n(z_1)|=\rho_n(z_1)$. Hence, also by Lemma~\ref{lem1}, $\rho_n(z_2)\geq |a_n(z_2)|$. 

We know 
\[\int_{U_P}|a_n|^2 \frac{i}{2} dz \wedge d\bar{z} \leq \| \varphi\| = 1 \, .\]
Since $U_P$ has finite area (with respect to $\frac{i}{2} dz \wedge d\bar{z}$) we conclude (by Cauchy--Schwarz) that
\[
\int_{U_P}|a_n| \frac{i}{2} dz \wedge d\bar{z} \leq \left(\int_{U_P} \frac{i}{2} dz \wedge d\bar{z} \right)^{\frac{1}{2}} \, .
\]

Suppose the $L^1$-norm of an analytic function $f$ is bounded by $B$
on the $\epsilon$ disk $D_\epsilon$ centered at the origin. Let
$D_{\epsilon/2}$ be the subdisk with radius $\epsilon/2$ again
centered at the origin. Let $z_0$ be any point in
$D_{\epsilon/2}$. Here, it does not matter whether $D_{\epsilon/2}$ is
closed or not. Then there exists a small disk centered at
$z_0$ with radius $\epsilon/4$ whose closure is completely contained
in $D_\epsilon$. Let $\gamma_{z_0}$ be the boundary of such a
disk. Then by the Cauchy's integral formula, 
$$|f'(z_0)| \leq \frac{1}{2\pi} \int_{\gamma_{z_0}}
\dfrac{|f(z)|}{|z-z_0|^2} \, |dz| \leq \frac{B}{2\pi} \int_{\gamma_{z_0}}
\dfrac{1}{|z-z_0|^2} \, |dz| \, . $$ 
Since $\gamma$ is always a round circle with fixed radius, the last
term is just a uniform constant. Therefore, this gives a uniform bound
on the $L^1$-norm of the derivative of $f$ on $D_{\epsilon/2}$. 
As a consequence, $f$ is Lipschitz on $D_{\epsilon/2}$ with Lipschitz constant uniformly
bounded. In our case, this implies that 
$a_n$ is Lipschitz on the $\epsilon/2$ subdisk, with Lipschitz constant $L$ independent from $n$.

Since $|a_n(z_1)|=\rho_n(z_1)$ and $\rho_n(z_2)\geq |a_n(z_2)|$ we conclude 
\[\rho_n(z_2)\geq \rho_n(z_1)-L|z_2-z_1| \, .\] 
By symmetry, we also have 
\[\rho_n(z_1)\geq \rho_n(z_2)-L|z_1-z_2| \, .\]

 \medskip

Now we show that the convergence of $\rho_n$ to $\rho_\infty$ is uniform on the disk centered at $0$ with radius $\epsilon/2$. If it is not, there must exist some $\epsilon'>0$, and a subsequence $\{\rho_{n_k}\}$ which is $\epsilon'$-away from $\rho_\infty$ under the uniform norm on the disk of radius $\epsilon/2$. Then, by the Arzel\`a--Ascoli theorem and the Claim, there exists a subsequence converging to some function uniformly, which is $\epsilon'$ away from $\rho_\infty$. This contradicts Theorem~\ref{thm:Kazhdan_surface}. 
\end{proof}

\input{GenKazh.bbl}

%\bibliography{Refs}
%\bibliographystyle{alpha}

\end{document}